\documentclass[11pt,oneside]{amsart}

\usepackage[margin=1in]{geometry}
\usepackage{amsmath,amssymb,amsthm}
\usepackage[table]{xcolor}
\usepackage{hyperref}

\hypersetup{
  colorlinks=true,
  linkcolor=blue,
  citecolor=red,
  urlcolor=blue
}

\newtheorem{theorem}{Theorem}[section]
\newtheorem{lemma}[theorem]{Lemma}
\newtheorem{corollary}[theorem]{Corollary}
\newtheorem{remark}[theorem]{Remark}

\DeclareMathOperator{\dg}{deg}
\newcommand{\cc}{\mathcal}

\begin{document}

\title{Extending Symmetric Layer-Rainbow Latin Cubes}
\author{Amin Bahmanian}
\address{Department of Mathematics, Illinois State University, Normal, IL USA 61790-4520}

\begin{abstract}
An $n\times n\times n$ array on $n^2$ symbols is a layer-rainbow Latin cube if every layer contains every symbol exactly once.  We call it symmetric if $L_{ij\ell}=L_{j\ell i}=L_{\ell ij}$ for distinct $i,j,\ell$ and $L_{iij}=L_{jji}$, $L_{iji}=L_{jij}$, $L_{ijj}=L_{jii}$ for distinct $i,j$.  We determine exactly when a symmetric layer-rainbow Latin cube of order $m$ embeds in one of order $n$, giving a three-dimensional analogue of Cruse's embedding theorem.  Call a positive integer admissible if it is congruent to $0$ or $2$ modulo $3$, with $1$ admissible and $3$ excluded.  For $n>m$, an embedding exists if and only if $m,n$ are admissible, $(m,n)\ne(2,5)$, and
\[
\begin{cases}
 n\geq2m,&n-m\not\equiv1\pmod3,\\[1mm]
 \displaystyle n\geq m+\frac{\sqrt{48m^2+1}-1}{6},&n-m\equiv1\pmod3.
\end{cases}
\]
Via the equivalent one-factorization problem for a non-uniform hypergraph, fair detachment reduces the proof to an exact integer allocation.  We also determine the structure forced at both sharp boundaries and obtain infinitely many equality cases.  At order eight, we construct a symmetric layer-rainbow Latin cube admitting the natural diagonal action of $\operatorname{PSL}(2,7)$, whose induced action on the $64$ symbols has orbit sizes $1,7,28,28$.
\end{abstract}

\maketitle

\section{Introduction}

We let $L$ be an $n\times n\times n$ array.  A \emph{layer} of $L$ is obtained by fixing one coordinate.  We call $L$ a \emph{layer-rainbow Latin cube} if it is filled with $n^2$ symbols and every layer contains every symbol exactly once.  Latin cubes and hypercubes were introduced by Kishen in 1942~\cite{Kishen1942}; in his later terminology, the objects just defined are precisely \emph{Latin cubes of the second order}~\cite{Kishen1949}.  Fisher subsequently studied related orthogonal cube constructions in experimental design~\cite{Fisher1945}; the statistical notion associated with Fisher uses $n$ symbols, each occurring $n$ times in every layer~\cite{HedayatSloaneStufken1999}.  Curiously, neither this notion nor Kishen's second-order notion coincides with a common modern combinatorial usage of \emph{Latin cube}: an $n\times n\times n$ array on $n$ symbols in which every axis-parallel line contains every symbol exactly once~\cite{CasselgrenMarkstromPham2019}.  Since several inequivalent notions have been called Latin cubes, we use the term \emph{layer-rainbow} to distinguish the present one.

Kishen's cubes were motivated by the design of experiments and admit a natural orthogonal-array interpretation with unequal column sizes~\cite{HedayatSloaneStufken1999}.  Indeed, in the $n^3$ rows $(i,j,k,L_{ijk})$, every ordered pair in two coordinate columns occurs $n$ times, while every coordinate-symbol pair occurs exactly once.  We will not need this interpretation here; the proof is most transparent in the equivalent hypergraph language developed in Section~2.

The symmetric version considered here was studied in~\cite{MR4665304}.  Following that paper, we call $L$ \emph{symmetric} if, for distinct $i,j,\ell$,
\[
 L_{ij\ell}=L_{j\ell i}=L_{\ell ij},
\]
and, for distinct $i,j$,
\[
 L_{iij}=L_{jji},\qquad L_{iji}=L_{jij},\qquad L_{ijj}=L_{jii}.
\]

As in a symmetric Latin square, where row $i$ determines column $i$ and conversely, the $i$th layer in any coordinate direction determines the $i$th layers in the other two: for distinct $i,j,k$, $L_{jki}=L_{ijk}$ and $L_{jik}=L_{ikj}$, with the repeated-coordinate cases covered by the identities above.

It was proved in~\cite[Theorem~1.1]{MR4665304} that a symmetric layer-rainbow Latin cube of order $n$ exists exactly when $n$ is congruent to $0$ or $2$ modulo $3$, with the exceptions that order $1$ exists and order $3$ does not.  We call such an order \emph{admissible}.

We study the corresponding embedding problem.  A cube of order $m$ \emph{embeds} in a cube of order $n$ if, after identifying its coordinate set with an $m$-subset of the new coordinate set and relabeling symbols if necessary, all entries on the old $m\times m\times m$ subcube are unchanged.  Throughout we assume $n>m$.  Without symmetry, a layer-rainbow cube of order $m$ embeds in one of order $n$ if and only if $n\geq2m$~\cite[Theorem~1.1]{BahmanianLayerRainbowEmbedding}; thus the classical Latin-square embedding threshold survives exactly in that three-dimensional setting.  The natural two-dimensional precedent for the symmetric problem is Cruse's theorem on embedding incomplete symmetric Latin squares~\cite{Cruse1974}.  From the hypergraph point of view, the problem also belongs to the theory of extending one-factorizations of complete designs: Cameron~\cite{Cameron1976} posed the general extension problem, Baranyai and Brouwer~\cite{BaranyaiBrouwer1977} proved the conjectured criterion for uniformities $2$ and $3$, and H\"aggkvist and Hellgren~\cite{HaggkvistHellgren1993} settled the general case.  Related rank-three embedding results appear in~\cite{Bahmanian2024Cruse}.

Symmetry changes the hypergraph itself.  Nonsymmetric layer-rainbow cubes correspond to one-factorizations of a complete tripartite $3$-uniform hypergraph, whereas symmetric layer-rainbow cubes correspond to one-factorizations of a hypergraph whose singleton, pair, and triple edges occur with multiplicities $1,3,2$, respectively.  These three edge sizes must therefore be handled together, and an additional obstruction can appear.

The symmetric problem turns out to retain the threshold $n\geq2m$ in most cases, but symmetry creates one additional obstruction.  We set $q=n-m$.  When $q\not\equiv1\pmod3$, the condition is simply $q\geq m$.  When $q\equiv1\pmod3$, the limited supply of pair edges forces the sharper inequality
\[
 4m^2\leq q(3q+1).
\]
There is also one exceptional pair, $(m,n)=(2,5)$.  Accordingly, we call $(m,n)$ \emph{feasible} if $m$ and $n$ are admissible, $(m,n)\ne(2,5)$, and
\[
\begin{cases}
 q\geq m,&q\not\equiv1\pmod3,\\[1mm]
 4m^2\leq q(3q+1),&q\equiv1\pmod3.
\end{cases}
\]
Equivalently, the second inequality is
\[
 n\geq m+\frac{\sqrt{48m^2+1}-1}{6}.
\]

Our main result is the following.

\begin{theorem}\label{thm:symmetric-cube-embedding}
A symmetric layer-rainbow Latin cube of order $m$ embeds in one of order $n$ if and only if $(m,n)$ is feasible.
\end{theorem}

\medskip
\noindent Table~\ref{tab:symembed26} gives the smallest nontrivial embedding.  The six matrices are the first-coordinate layers of the order-$6$ cube.  The shaded $2\times2$ blocks in the first two layers form the embedded order-$2$ cube.

\begin{table}[htbp]
\centering
\caption{A symmetric layer-rainbow Latin cube of order six containing one of order two.  The shaded cells form the embedded cube.}
\label{tab:symembed26}
\setlength{\tabcolsep}{2.2pt}
\renewcommand{\arraystretch}{0.92}
\scriptsize

\begin{minipage}[t]{0.30\textwidth}
\centering
\textit{Layer 1}\\[2pt]
\begin{tabular}{@{}|c|c|c|c|c|c|@{}}
\hline
\cellcolor{gray!30} A & \cellcolor{gray!30} B & E & H & K & N \\ \hline
\cellcolor{gray!30} C & \cellcolor{gray!30} D & Q & S & U & W \\ \hline
F & R & G & Y & 0 & 2 \\ \hline
I & T & Z & J & 4 & 6 \\ \hline
L & V & 1 & 5 & M & 8 \\ \hline
O & X & 3 & 7 & 9 & P \\ \hline
\end{tabular}
\end{minipage}
\hfill
\begin{minipage}[t]{0.30\textwidth}
\centering
\textit{Layer 2}\\[2pt]
\begin{tabular}{@{}|c|c|c|c|c|c|@{}}
\hline
\cellcolor{gray!30} D & \cellcolor{gray!30} C & R & T & V & X \\ \hline
\cellcolor{gray!30} B & \cellcolor{gray!30} A & H & L & E & F \\ \hline
Q & I & K & 8 & 6 & 4 \\ \hline
S & M & 9 & N & 2 & 0 \\ \hline
U & O & 7 & 3 & P & Y \\ \hline
W & G & 5 & 1 & Z & J \\ \hline
\end{tabular}
\end{minipage}
\hfill
\begin{minipage}[t]{0.30\textwidth}
\centering
\textit{Layer 3}\\[2pt]
\begin{tabular}{@{}|c|c|c|c|c|c|@{}}
\hline
G & Q & F & Z & 1 & 3 \\ \hline
R & K & I & 9 & 7 & 5 \\ \hline
E & H & A & B & C & D \\ \hline
Y & 8 & O & P & W & U \\ \hline
0 & 6 & J & X & N & S \\ \hline
2 & 4 & L & V & T & M \\ \hline
\end{tabular}
\end{minipage}

\vspace{0.35em}

\begin{minipage}[t]{0.30\textwidth}
\centering
\textit{Layer 4}\\[2pt]
\begin{tabular}{@{}|c|c|c|c|c|c|@{}}
\hline
J & S & Y & I & 5 & 7 \\ \hline
T & N & 8 & M & 3 & 1 \\ \hline
Z & 9 & P & O & X & V \\ \hline
H & L & B & A & D & C \\ \hline
4 & 2 & W & F & G & Q \\ \hline
6 & 0 & U & E & R & K \\ \hline
\end{tabular}
\end{minipage}
\hfill
\begin{minipage}[t]{0.30\textwidth}
\centering
\textit{Layer 5}\\[2pt]
\begin{tabular}{@{}|c|c|c|c|c|c|@{}}
\hline
M & U & 0 & 4 & L & 9 \\ \hline
V & P & 6 & 2 & O & Z \\ \hline
1 & 7 & N & W & J & T \\ \hline
5 & 3 & X & G & F & R \\ \hline
K & E & C & D & A & B \\ \hline
8 & Y & S & Q & H & I \\ \hline
\end{tabular}
\end{minipage}
\hfill
\begin{minipage}[t]{0.30\textwidth}
\centering
\textit{Layer 6}\\[2pt]
\begin{tabular}{@{}|c|c|c|c|c|c|@{}}
\hline
P & W & 2 & 6 & 8 & O \\ \hline
X & J & 4 & 0 & Y & G \\ \hline
3 & 5 & M & U & S & L \\ \hline
7 & 1 & V & K & Q & E \\ \hline
9 & Z & T & R & I & H \\ \hline
N & F & D & C & B & A \\ \hline
\end{tabular}
\end{minipage}
\end{table}

When $n-m\not\equiv1\pmod3$, symmetry leaves the unsymmetric threshold $n\geq2m$ unchanged.  When $n-m\equiv1\pmod3$, the sharp threshold is asymptotically $(1+2/\sqrt3)m$.  At the sharp bounds, the embedding is highly constrained: in the exceptional congruence case the old symbols occupy every cell of the new subcube with at least two equal coordinates, while at $n=2m$ they fill the entire opposite $m\times m\times m$ subcube, with the new symbols balanced between the two types of mixed regions.  We identify all old vertices with one vertex and all new vertices with another, solve the resulting finite integer allocation problem, and then apply fair detachment.  The detachment approach used here was inspired by work of Hilton and Rodger on graph decompositions and by Ferencak and Hilton on amalgamated triple systems~\cite{HiltonRodger1986Hamilton,FerencakHilton2002}.

At order eight, the profiles from Lemma~\ref{lem:singleton-factor} admit a much more rigid realization under the natural action of $G=\operatorname{PSL}(2,7)$.  In Section~\ref{sec:psl27} we prove that there is a unique $G$-invariant one-factorization of $K_8^2\cup K_8^3$, necessarily with factor type $2+3+3$, and extend it to a $G$-invariant one-factorization of $\cc H_8$ with orbit sizes $1,7,28,28$.

\section{Hypergraph formulation and the embedding theorem}

Throughout, hypergraphs may have repeated edges, and an edge may contain a vertex more than once.  The degree $\dg_{\cc F}(v)$ of a vertex $v$ is the number of incidences of $v$ with the edges of $\cc F$, counting repetitions.  If $\cc F$ is edge-colored, then $\cc F(j)$ denotes the spanning subhypergraph whose edges are those of color $j$.  A \emph{one-factor} is a spanning $1$-regular subhypergraph, and a \emph{one-factorization} is an edge-coloring in which every color class is a one-factor.  The \emph{profile} of a color class is $(s,p,t)$, where $s$, $p$, and $t$ are its numbers of singleton, pair, and triple edges, respectively.  If $Y$ is a multiset of vertices of $\cc F$, we let $(Y)_{\cc F}$ denote the number of edges equal to $Y$, counting multiplicity.  When the underlying hypergraph is clear from context, we use $(Y)$.

For positive integers $n$ and $r$, let $K_n^r$ denote the complete $r$-uniform hypergraph on $n$ vertices.  We let
\[
 \cc H_n=K_n^1\cup3K_n^2\cup2K_n^3.
\]
Every vertex of $\cc H_n$ has degree $n^2$, so every one-factorization of $\cc H_n$ has $n^2$ colors.

The following correspondence is the hypergraph formulation of symmetric layer-rainbow Latin cubes from~\cite[Lemma~2.1]{MR4665304}.

\begin{lemma}\label{symcubequi1fac}
There is a one-to-one correspondence between symmetric layer-rainbow Latin cubes of order $n$ and one-factorizations of $\cc H_n$.
\end{lemma}

We say that a one-factorization of $\cc H_m$ embeds in one of $\cc H_n$ if, after identifying $V(\cc H_m)$ with an $m$-subset of $V(\cc H_n)$ and relabeling colors if necessary, the latter restricts to the given one-factorization on $\cc H_m$.

We use the following detachment lemma~\cite{BahmanianConnectedFair}.  For an integer $x$ and a real number $y$, we let $x\approx y$ mean $x\in\{\lfloor y\rfloor,\lceil y\rceil\}$.

\begin{lemma}\label{splittinglemma}
Suppose $\cc F$ is an edge-colored hypergraph on $\{\alpha,\beta\}$, $a,b$ are positive integers, and no edge contains $\alpha$ more than $a$ times or $\beta$ more than $b$ times.  Then $\alpha$ and $\beta$ can be split into disjoint sets $A,B$ of sizes $a,b$ to obtain a hypergraph $\cc F'$ in which no edge contains a vertex more than once.  For a color $j$, $x\in A$, and $y\in B$,
\[
 \dg_{\cc F'(j)}(x)\approx\frac{\dg_{\cc F(j)}(\alpha)}a,
 \qquad
 \dg_{\cc F'(j)}(y)\approx\frac{\dg_{\cc F(j)}(\beta)}b.
\]
For $U\subseteq A$ and $W\subseteq B$ with $U\cup W\ne\varnothing$,
\[
 (U\cup W)_{\cc F'}\approx
 \frac{(\alpha^{|U|}\beta^{|W|})_{\cc F}}
 {\binom a{|U|}\binom b{|W|}}.
\]
\end{lemma}

We shall also use the immediate one-vertex form: if no edge of an edge-colored hypergraph on one vertex $\alpha$ contains $\alpha$ more than $a$ times, then $\alpha$ may be split fairly into $a$ vertices.  This follows by adding an isolated dummy vertex and taking the second split size to be $1$ in Lemma~\ref{splittinglemma}.

\medskip
\noindent\textit{The two-vertex reduction.}
We let a one-factorization of $\cc H_m$ be given and set $q=n-m$.  We let $(\lambda_1,\lambda_2,\lambda_3)=(1,3,2)$, and we let $\cc G$ be the hypergraph on $\{\alpha,\beta\}$ in which, for integers $r\geq0$ and $s\geq1$ with $r+s\leq3$,
\[
 (\alpha^r\beta^s)_{\cc G}=\lambda_{r+s}\binom mr\binom qs,
\]
and no other edges occur.  Equivalently,
\[
\begin{array}{lll}
(\beta)=q,&(\beta^2)=3\binom q2,&(\beta^3)=2\binom q3,\\[1mm]
(\alpha\beta)=3mq,&(\alpha^2\beta)=qm(m-1),&(\alpha\beta^2)=mq(q-1).
\end{array}
\]
We call the $m^2$ colors of the given one-factorization \emph{old} and the remaining $n^2-m^2$ colors \emph{new}.  For a color $i$ and a vertex $v$, we let $\dg_i(v)$ denote the current degree of $v$ in color $i$.  After some edges have been colored, $(\alpha^r\beta^s)_i$ denotes the number of edges of type $\alpha^r\beta^s$ colored with $i$.  All edges of a fixed type are parallel in $\cc G$, so whenever nonnegative target numbers for one edge type have the correct total, the corresponding edges may be colored accordingly.

The given one-factorization of $\cc H_m$ extends to a one-factorization of $\cc H_n$ precisely when the edges of $\cc G$ can be colored so that
\[
 (\dg_i(\alpha),\dg_i(\beta))=
 \begin{cases}
 (0,q),&i\text{ old},\\
 (m,q),&i\text{ new}.
 \end{cases}
\]
To see this, first suppose that the given one-factorization extends.  Delete the copy of $\cc H_m$ on the old vertices, identify the $m$ old vertices with $\alpha$, and identify the $q$ new vertices with $\beta$.  For each $r\geq0$ and $s\geq1$ with $r+s\leq3$, there are $\binom mr\binom qs$ sets containing exactly $r$ old and $s$ new vertices, each occurring $\lambda_{r+s}$ times, so the resulting hypergraph is $\cc G$.  An old color has no added edge meeting an old vertex, whereas a new color covers every old and every new vertex once.  The stated degree conditions follow.

Conversely, suppose that $\cc G$ has such a coloring.  Every edge occurring in $\cc G$ contains $\alpha$ at most $m$ times and $\beta$ at most $q$ times, so Lemma~\ref{splittinglemma} splits $\alpha$ and $\beta$ simultaneously into disjoint sets $A$ and $B$ of sizes $m$ and $q$, producing $\cc G'$.  If $U\subseteq A$ and $\varnothing\ne W\subseteq B$, let $r=|U|$ and $s=|W|$.  Whenever $r+s\leq3$, Lemma~\ref{splittinglemma} gives
\[
 (U\cup W)_{\cc G'}\approx
 \frac{(\alpha^r\beta^s)_{\cc G}}{\binom mr\binom qs}
 =\lambda_{r+s}.
\]
No edge lies entirely in $A$, because $\cc G$ has no edge containing only $\alpha$.  Hence the detached hypergraph is exactly $\cc H_n$ with the copy of $\cc H_m$ on $A$ deleted.

The degree conclusion of Lemma~\ref{splittinglemma} gives degree $0$ at every vertex of $A$ and degree $1$ at every vertex of $B$ in each old color, and degree $1$ at every vertex of $A\cup B$ in each new color.  We restore the given one-factorization of $\cc H_m$ on $A$ using the old colors.  Every color now has degree $1$ at every vertex of $A\cup B$, giving a one-factorization of $\cc H_n$ that extends the given one.

\begin{lemma}\label{lem:necessity}
If a one-factorization of $\cc H_m$ extends to a one-factorization of $\cc H_n$, then $(m,n)$ is feasible.
\end{lemma}

\begin{proof}
Set $q=n-m$.  Necessarily $m$ and $n$ are admissible.  Each old color already covers every old vertex, so every added edge of an old color lies entirely on the $q$ new vertices.  For an old color $i$, let $s_i,p_i,t_i$ be the numbers of $\beta$-, $\beta^2$-, and $\beta^3$-edges, respectively.  Then $s_i+2p_i+3t_i=q$.

Counting $\beta$-incidences in the $m^2$ old colors gives
\[
 m^2q\leq q+6\binom q2+6\binom q3=q^3,
\]
and hence $q\geq m$.

Suppose now that $q\equiv1\pmod3$.  If $s_i=0$, then $2p_i\equiv1\pmod3$, so $p_i\equiv2\pmod3$ and therefore $p_i\geq2$.  At most $q$ old colors contain a $\beta$-edge, while every remaining old color requires at least two $\beta^2$-edges.  Consequently
\[
 m^2\leq q+\left\lfloor\frac{3\binom q2}{2}\right\rfloor
 =\left\lfloor\frac{q(3q+1)}4\right\rfloor,
\]
so $4m^2\leq q(3q+1)$.

It remains to exclude $(m,n)=(2,5)$.  Here $q=3$.  For a new color, let $x,y,z$ be the numbers of $\alpha\beta$-, $\alpha^2\beta$-, and $\alpha\beta^2$-edges.  Its required degrees give $x+2y+z=2$ and $x+y+2z\leq3$, whose only solutions are $(2,0,0)$, $(1,0,1)$, and $(0,1,0)$.  Since there are respectively $18$, $6$, and $12$ edges of the three types involving both $\alpha$ and $\beta$, the $21$ new colors have these types in numbers $3$, $12$, and $6$.

The three colors of type $(2,0,0)$ require all three $\beta$-edges, while each of the six colors of type $(0,1,0)$ requires a $\beta^2$-edge.  Thus no $\beta$-edge remains for an old color.  Each of the four old colors must consequently contain a $\beta^3$-edge in order to have degree $3$ at $\beta$, but only two such edges exist, a contradiction.  This excludes $(m,n)=(2,5)$ and completes the proof.
\end{proof}

By Lemma~\ref{symcubequi1fac}, Theorem~\ref{thm:symmetric-cube-embedding} is equivalent to the following statement.

\begin{theorem}\label{thm:hypergraph-embedding}
A one-factorization of $\cc H_m$ embeds in one of $\cc H_n$ if and only if $(m,n)$ is a feasible pair.
\end{theorem}

From now on, we set $q=n-m$ and assume that $(m,n)$ is feasible.  For a real number $x$, we let $x^+=\max\{0,x\}$.  We shall use repeatedly that $q\geq m$.  This is part of feasibility when $q\not\equiv1\pmod3$.  If $q\equiv1\pmod3$ and $q<m$, then $q(3q+1)\leq(m-1)(3m-2)<4m^2$, contrary to feasibility.

\medskip
\noindent\textit{The order-two case.}
A one-factorization of $\cc H_2$ is unique up to relabeling colors: the three parallel pair edges are three one-factors, and the two singleton edges form the fourth.  Thus, for $m=2$, it is enough to find a one-factorization of $\cc H_n$ in which all singleton edges lie in one factor.

\begin{lemma}\label{lem:singleton-factor}
If $n\geq6$ is admissible, then $\cc H_n$ has a one-factorization in which all $n$ singleton edges form one factor.
\end{lemma}

\begin{proof}
Let $\cc F$ be obtained from $\cc H_n$ by identifying its $n$ vertices with a single vertex $\alpha$.  We color the edges of $\cc F$ according to the following profiles; the first column gives the number of colors of each profile.

{\small
\[
\begin{array}{c@{\qquad}c}
\begin{array}{c|ccc}
\multicolumn{4}{c}{n\equiv0\pmod3}\\[1mm]
\text{number}&(\alpha)_i&(\alpha^2)_i&(\alpha^3)_i\\ \hline
1&n&0&0\\[1mm]
\dfrac{(n-1)(n+2)}2&0&0&\dfrac n3\\[2mm]
\dfrac{n(n-1)}2&0&3&\dfrac{n-6}{3}
\end{array}
&
\begin{array}{c|ccc}
\multicolumn{4}{c}{n\equiv2\pmod3}\\[1mm]
\text{number}&(\alpha)_i&(\alpha^2)_i&(\alpha^3)_i\\ \hline
1&n&0&0\\[1mm]
\dfrac{(n-1)(5n+8)}6&0&1&\dfrac{n-2}{3}\\[2mm]
\dfrac{(n-1)(n-2)}6&0&4&\dfrac{n-8}{3}
\end{array}
\end{array}
\]
}

All entries are nonnegative integers, and each row has degree $n$ at $\alpha$.  In either case the row multiplicities sum to $n^2$, while the three edge-type columns use exactly $n$, $3\binom n2$, and $2\binom n3$ edges, respectively.  Thus every edge of $\cc F$ is colored exactly once, every color has degree $n$ at $\alpha$, and all $\alpha$-edges have the same color.

Apply the one-vertex form of Lemma~\ref{splittinglemma} to split $\alpha$ into a set $A$ of $n$ vertices, producing $\cc F'$.  For every color $i$, every $x\in A$, and every nonempty $U\subseteq A$ with $|U|\leq3$,
\[
 \dg_{\cc F'(i)}(x)\approx\frac{\dg_{\cc F(i)}(\alpha)}n=1,
 \qquad
 (U)_{\cc F'}\approx
 \frac{(\alpha^{|U|})_{\cc F}}{\binom n{|U|}}
 =
 \begin{cases}
 1,&|U|=1,\\
 3,&|U|=2,\\
 2,&|U|=3.
 \end{cases}
\]
Hence the resulting hypergraph is $\cc H_n$, every color class is a one-factor, and since all $\alpha$-edges had the same color before the split, all singleton edges of $\cc H_n$ lie in one factor.
\end{proof}

\medskip
\noindent\textit{The general construction for $m\geq5$.}
We now return to the two-vertex hypergraph $\cc G$ defined above.  Throughout Lemmas~\ref{lem:completion} and~\ref{lem:partial-coloring}, all edges, edge types, and degrees refer to $\cc G$.  The next lemma shows how to finish the coloring once only the $\beta^2$- and $\beta^3$-edges remain.

\begin{lemma}\label{lem:completion}
Suppose $\cc G$ has been partially colored so that the only uncolored edges are those of types $\beta^2$ and $\beta^3$.  For each color $i$, put $r_i=q-\dg_i(\beta)$.  Suppose there are nonnegative integers $a_i,b_i$ such that $r_i=2a_i+3b_i$, with $\sum_i a_i\leq3\binom q2$, and with at most $2\binom q3$ of the $b_i$ odd.  Then the coloring extends to all of $\cc G$ so that $\dg_i(\beta)=q$ for every color $i$.
\end{lemma}

\begin{proof}
Think of $a_i$ and $b_i$ as tentative numbers of $\beta^2$- and $\beta^3$-edges for color $i$.  They supply the required degree $r_i=2a_i+3b_i$, but their totals need not equal the numbers of edges available.  The key observation is that replacing two $\beta^3$-edges by three $\beta^2$-edges preserves degree at $\beta$.

Summing the deficits over all colors gives $2\sum_i a_i+3\sum_i b_i=6\binom q2+6\binom q3$.  Hence
\[
 d=\binom q2-\frac13\sum_i a_i
\]
is a nonnegative integer, and $\sum_i b_i=2\binom q3+2d$.  Thus exactly $d$ replacements are needed.

Let $o$ be the number of odd $b_i$.  Since $o\leq2\binom q3$,
\[
 \sum_i\left\lfloor\frac{b_i}{2}\right\rfloor
 =\frac{\sum_i b_i-o}{2}
 =d+\binom q3-\frac{o}{2}\geq d.
\]
We may therefore choose integers $0\leq t_i\leq\lfloor b_i/2\rfloor$ with $\sum_i t_i=d$.  Set $u_i=a_i+3t_i$ and $v_i=b_i-2t_i$.  Then
\[
 \sum_i u_i=3\binom q2,\qquad
 \sum_i v_i=2\binom q3,\qquad
 2u_i+3v_i=r_i.
\]
Color $u_i$ of the remaining $\beta^2$-edges and $v_i$ of the remaining $\beta^3$-edges with color $i$.  These totals use every remaining edge exactly once, and every color finishes with degree $q$ at $\beta$.
\end{proof}

\noindent The main allocation is the following.

\begin{lemma}\label{lem:partial-coloring}
Suppose $m\geq5$ and $(m,n)$ is feasible.  All edges of $\cc G$ of types $\beta$, $\alpha\beta$, $\alpha^2\beta$, and $\alpha\beta^2$ can be colored so that $\dg_i(\alpha)=0$ for old colors and $\dg_i(\alpha)=m$ for new colors.  Moreover, for every color $i$, the remaining deficit $r_i=q-\dg_i(\beta)$ can be written as $r_i=2a_i+3b_i$ with $a_i,b_i\geq0$, where $\sum_i a_i\leq3\binom q2$ and at most $2\binom q3$ of the $b_i$ are odd.
\end{lemma}

\begin{proof}

By Lemma~\ref{lem:completion}, it is enough to color the $\beta$-, $\alpha\beta$-, $\alpha^2\beta$-, and $\alpha\beta^2$-edges so that the remaining degree satisfies $r_i=2a_i+3b_i$, with enough $\beta^2$-edges available and sufficiently few odd $b_i$.  We do this in four parts: old colors, the initial new-color table, the $\alpha^2\beta$- and $\alpha\beta^2$-edge allocation, and parity control.  We first assume $(m,q)\ne(5,7)$; this case is handled separately at the end.

\noindent\textit{Old colors.}  We set $\delta=(q-m^2)^+$.  For the old colors, choose $q-\delta$ of the $\beta$-edges and color them with distinct old colors, one edge with each color.  Leave the remaining $\delta$ $\beta$-edges uncolored for the moment.  Thus $(\beta)_i=1$ for the old colors used and $(\beta)_i=0$ for the others.  For each old color $i$, let $a_i\in\{0,1,2\}$ be the unique integer such that $q-(\beta)_i-2a_i\equiv0\pmod3$, and set
\[
 b_i=\frac{q-(\beta)_i-2a_i}{3}.
\]
Then $b_i$ is a nonnegative integer: integrality follows from the choice of $a_i$, and $b_i\geq0$ because $q-(\beta)_i-2a_i\geq q-5\geq0$.  Since an old color receives no edge containing $\alpha$, it has degree $0$ at $\alpha$, and
\[
 q-\dg_i(\beta)=q-(\beta)_i=2a_i+3b_i.
\]
For the remainder of the proof, $i$ denotes a new color, and all sums are over the new colors unless explicitly stated otherwise.

\noindent\textit{Initial coloring of the new colors.}  For a new color $i$, let $x_i$ denote the number of $\alpha^2\beta$-edges that will eventually receive color $i$.  Once $x_i$ is chosen, the requirement $\dg_i(\alpha)=m$ forces the number of $\alpha\beta^2$-edges to be $y_i=m-(\alpha\beta)_i-2x_i$.  No edge of either of these two types is colored yet.  The interval $[\iota_i,\rho_i]$ in the table records the range in which $x_i$ will be allowed to vary: the upper endpoint guarantees $y_i\geq0$, while the lower endpoint guarantees that the remaining degree needed at $\beta$ has the form required for the final completion.

We first color the $\beta$- and $\alpha\beta$-edges according to the following table.  For each color in a row, color the indicated number of currently uncolored edges of these two types with that color, take the indicated value of $a_i$, and record the interval $[\iota_i,\rho_i]$.  For the $\delta$ colors in the first row in either case, color one $\beta$-edge with each color; no other new color receives a $\beta$-edge.

Since $m\equiv0$ or $2\pmod3$, we have $m^2\equiv2m\pmod3$.  Thus, when $n\equiv0\pmod3$, $q-m^2\equiv q-2m\equiv0\pmod3$, and hence $3\mid\delta$.  Also, $3\mid(2m-q)$ when $n\equiv0\pmod3$, while $3\mid(2m+2-q)$ when $n\equiv2\pmod3$.  Hence all row counts and all lower endpoints $\iota_i$ in the following table are integers.  The number in the second column is the number of new colors of the indicated type.
\[
\begin{array}{c|c|ccc|cc}
&\text{number}&(\beta)_i&(\alpha\beta)_i&a_i&\iota_i&\rho_i\\ \hline
&\delta
&1&1&0
&\displaystyle\left(\frac{2m-q}{3}\right)^+
&\displaystyle\left\lfloor\frac{m-1}{2}\right\rfloor\\[3mm]
n\equiv0\pmod3
&\displaystyle mq-\frac{\delta}{3}
&0&3&0
&\displaystyle\left(\frac{2m-q}{3}-1\right)^+
&\displaystyle\left\lfloor\frac{m-3}{2}\right\rfloor\\[3mm]
&\displaystyle q(q+m)-\frac{2\delta}{3}
&0&0&0
&\displaystyle\left(\frac{2m-q}{3}\right)^+
&\displaystyle\left\lfloor\frac m2\right\rfloor\\[2mm] \hline
&\delta
&1&1&1
&\displaystyle\left(\frac{2m+2-q}{3}\right)^+
&\displaystyle\left\lfloor\frac{m-1}{2}\right\rfloor\\[3mm]
n\equiv2\pmod3
&3mq-\delta
&0&1&0
&\displaystyle\left(\frac{2m+2-q}{3}-1\right)^+
&\displaystyle\left\lfloor\frac{m-1}{2}\right\rfloor\\[3mm]
&q(q-m)
&0&0&1
&\displaystyle\left(\frac{2m+2-q}{3}\right)^+
&\displaystyle\left\lfloor\frac m2\right\rfloor.
\end{array}
\]
In every row, $(\alpha\beta)_i+2\rho_i\leq m$, so $x_i\leq\rho_i$ will indeed imply $y_i\geq0$.

In each case the row counts are nonnegative integers and sum to $q(q+2m)=n^2-m^2$, so they partition the new colors as claimed.  Here we use only $0\leq\delta\leq q$ and $q\geq m$; summing the rows gives $q(q+2m)$.  The $\beta$-column has total $\delta$, so together with the $q-\delta$ $\beta$-edges already colored with old colors, every $\beta$-edge is colored exactly once.  Likewise, the $\alpha\beta$-column sums to $3mq$, so every $\alpha\beta$-edge is colored exactly once.

Among the old colors, $q-\delta=\min\{q,m^2\}$ have $(\beta)_i=1$.  According as $q\equiv0,1,2\pmod3$, the values of $a_i$ for $(\beta)_i=1$ and $(\beta)_i=0$ are $(1,0)$, $(0,2)$, and $(2,1)$, respectively.  The new colors contribute $0$ when $n\equiv0\pmod3$, and $\delta+q(q-m)$ when $n\equiv2\pmod3$.  Hence
\[
\sum_i a_i=
\begin{cases}
\min\{q,m^2\},
&n\equiv0,\ q\equiv0\pmod3,\\[1mm]
2\bigl(m^2-\min\{q,m^2\}\bigr),
&n\equiv0,\ q\equiv1\pmod3,\\[1mm]
q(q-m+1),
&n\equiv2,\ q\equiv0\pmod3,\\[1mm]
m^2+q(q-m+1),
&n\equiv2,\ q\equiv2\pmod3.
\end{cases}
\]
These are the only possible residue pairs, since $n\equiv q-2m\pmod3$ and $m\equiv0$ or $2\pmod3$.  We verify that this total is at most $3\binom q2$.  If $n\equiv0$ and $q\equiv0\pmod3$, then $\sum_i a_i\leq q\leq3\binom q2$; if $n\equiv2$ and $q\equiv0\pmod3$, then $3\binom q2-\sum_i a_i=q(q+2m-5)/2\geq0$.  Suppose $n\equiv0$ and $q\equiv1\pmod3$.  If $q<m^2$, the inequality $\sum_i a_i\leq3\binom q2$ is exactly $4m^2\leq q(3q+1)$, which is the feasibility condition for this residue class; if $q\geq m^2$, the corresponding term is zero.  Finally, suppose $n\equiv2$ and $q\equiv2\pmod3$.  The inequality $\sum_i a_i\leq3\binom q2$ is equivalent to $q^2+(2m-5)q-2m^2\geq0$.  Here $m\equiv0\pmod3$, so $m\geq6$, and $q\equiv m+2\pmod3$ together with $q\geq m$ gives $q\geq m+2$.  The left-hand side is increasing in $q$, and at $q=m+2$ it equals $m^2+3m-6>0$.

The intervals in the table are nonempty.  If $n\equiv0\pmod3$, then $(2m-q)/3\leq m/3\leq(m-1)/2$, which verifies the first and third rows, while the middle row follows from $((2m-q)/3-1)^+\leq(m/3-1)^+\leq\lfloor(m-3)/2\rfloor$.  If $n\equiv2\pmod3$, then $q>m$ and $(2m+2-q)/3\leq(m+1)/3\leq(m-1)/2$, which verifies the first and third rows, while the middle row follows from $((2m+2-q)/3-1)^+\leq((m+1)/3-1)^+\leq\lfloor(m-1)/2\rfloor$.

\noindent\textit{The $\alpha^2\beta$- and $\alpha\beta^2$-edge allocation.}  We next verify that the intervals allow the required total for the $\alpha^2\beta$-allocation.  Let $T=qm(m-1)$, the total number of $\alpha^2\beta$-edges.  If $\sum_i\iota_i>0$, direct summation gives
\[
 T-\sum_i\iota_i
 =
\begin{cases}
\displaystyle
\frac{q(q^2-m^2)-\delta}{3},
&n\equiv0\pmod3,\\[3mm]
\displaystyle
\frac{q(q-m)(q+m-2)-3\delta}{3},
&n\equiv2\pmod3.
\end{cases}
\]
For $n\equiv0\pmod3$, this difference is nonnegative because either $q=m$ and $\delta=0$, or $q>m$ and $q(q^2-m^2)\geq q\geq\delta$.  For $n\equiv2\pmod3$, it is nonnegative because $q>m$ and $q(q-m)(q+m-2)\geq q(2m-1)\geq9q\geq3\delta$.  Thus $\sum_i\iota_i\leq T$.

We next prove the stronger upper estimate
\[
 \sum_i\rho_i-T\geq2mq.
\]
Direct summation gives
\[
\sum_i\rho_i-T-2mq
=
\begin{cases}
\displaystyle
\frac{mq(q-6)}2-\frac{\delta}{3},
&n\equiv0,\ m\text{ even},\\[3mm]
\displaystyle
q\left(\frac{q(m-1)}2-3m\right)+\frac{\delta}{3},
&n\equiv0,\ m\text{ odd},\\[3mm]
\displaystyle
\frac{mq(q-8)}2,
&n\equiv2,\ m\text{ even},\\[3mm]
\displaystyle
\frac q2\bigl(q(m-1)-4m\bigr),
&n\equiv2,\ m\text{ odd}.
\end{cases}
\]
If $n\equiv0\pmod3$ and $m$ is even, then $m\geq6$.  If $q=6$, feasibility forces $m=6$, so $\delta=0$ and the first expression is $0$.  Otherwise feasibility and the congruence conditions give $q\geq9$; since $\delta\leq q$, the first expression is nonnegative.  If $n\equiv0\pmod3$ and $m$ is odd, then, because $(m,q)\ne(5,7)$, either $m=5$ and $q\geq10$, or $m\geq9$ and $q\geq m$, so the second expression is nonnegative.  If $n\equiv2\pmod3$ and $m$ is even, admissibility gives $q\geq8$, so the third expression is nonnegative.  Finally, if $n\equiv2\pmod3$ and $m$ is odd, then $q\geq m\geq5$, so the fourth expression is at least $qm(m-5)/2\geq0$.  Hence
\[
 \sum_i\iota_i\leq T\leq\sum_i\rho_i.
\]

At this stage no $\alpha^2\beta$- or $\alpha\beta^2$-edge has been colored.  We now choose their multiplicities together.  For each new color $i$, let $x_i$ denote the number of $\alpha^2\beta$-edges that will receive color $i$.  Since a new color must have degree $m$ at $\alpha$, once $x_i$ is chosen the number of $\alpha\beta^2$-edges that must receive color $i$ is forced to be
\[
 y_i=m-(\alpha\beta)_i-2x_i.
\]
Thus the intervals $[\iota_i,\rho_i]$ in the table are precisely the allowed ranges for $x_i$: the lower endpoint is chosen so that the remaining degree at $\beta$ has the required form below, while the upper endpoint guarantees $y_i\geq0$.

A choice $x=(x_i)$ is an \emph{allocation} if $\iota_i\leq x_i\leq\rho_i$ for every new color $i$, and it is \emph{complete} if
\[
 \sum_i x_i=T=qm(m-1).
\]
For every complete allocation, the forced values $y_i$ automatically have the correct total:
\[
 \sum_i y_i=mq(q+2m)-3mq-2qm(m-1)=mq(q-1),
\]
which is exactly the number of $\alpha\beta^2$-edges.  Thus a complete allocation determines all $\alpha^2\beta$- and $\alpha\beta^2$-multiplicities.  No edge of either type is colored until such an allocation has been found.

\noindent\textit{Controlling parity.}  It remains to choose a complete allocation for which the still-unfilled degree at $\beta$ can be supplied by the $\beta^2$- and $\beta^3$-edges.  Let
\[
 \theta=
 \begin{cases}
 (2m-q)/3,&n\equiv0\pmod3,\\[1mm]
 (2m+2-q)/3,&n\equiv2\pmod3.
 \end{cases}
\]
For an integer $x\in[\iota_i,\rho_i]$, put
\[
 y_i(x)=m-(\alpha\beta)_i-2x
\]
and
\[
 c_i(x)=\frac{q-(\beta)_i-(\alpha\beta)_i-x-2y_i(x)-2a_i}{3}
       =\frac13\bigl(q-(\beta)_i-2m+(\alpha\beta)_i-2a_i\bigr)+x.
\]
If $x_i=x$, then color $i$ is to receive $x$ edges of type $\alpha^2\beta$ and $y_i(x)$ edges of type $\alpha\beta^2$, so the degree still required at $\beta$ is
\[
 q-\bigl((\beta)_i+(\alpha\beta)_i+x+2y_i(x)\bigr)
 =2a_i+3c_i(x).
\]
Hence the parity of $c_i(x)$ is the only remaining issue: Lemma~\ref{lem:completion} will distribute the $\beta^2$- and $\beta^3$-edges provided that not too many of the resulting values $c_i(x_i)$ are odd.

For the first, middle, and third rows corresponding to either value of $n\pmod3$, respectively, direct substitution gives
\[
 c_i(x)=x-\theta,\qquad x-\theta+1,\qquad x-\theta,
\]
and the corresponding lower endpoints are $\theta^+$, $(\theta-1)^+$, and $\theta^+$.  Thus $c_i(x)$ is a nonnegative integer throughout $[\iota_i,\rho_i]$, and $c_i(x+1)=c_i(x)+1$.

Let $\iota_i'$ and $\rho_i'$ be the least and greatest values of $x\in[\iota_i,\rho_i]$ for which $c_i(x)$ is even.  These values always exist.  Indeed, if $\iota_i>0$, the preceding formulas give $c_i(\iota_i)=0$.  If $\iota_i=0$, then one of $c_i(0)$ and $c_i(1)$ is even, and $1\leq\rho_i$ because the smallest upper endpoint in the table is $\lfloor(m-3)/2\rfloor\geq1$.  Since the parity changes at every step,
\[
 \iota_i'\in\{\iota_i,\iota_i+1\},
 \qquad
 \rho_i'\in\{\rho_i-1,\rho_i\},
\]
and the values of $x$ for which $c_i(x)$ is even are precisely
\[
 \iota_i',\ \iota_i'+2,\ \ldots,\ \rho_i'.
\]

Put
\[
 I=\sum_i\iota_i',
 \qquad
 R=\sum_i\rho_i'.
\]
Since each $\rho_i'-\iota_i'$ is even, $I\equiv R\pmod2$.  Moreover, for every integer $S$ of this parity with $I\leq S\leq R$, there is an allocation $x$ with $\sum_i x_i=S$ and every $c_i(x_i)$ even.  Indeed, write $x_i=\iota_i'+2z_i$ with $0\leq z_i\leq(\rho_i'-\iota_i')/2$; the required excess $(S-I)/2$ can be distributed among the $z_i$ greedily.

Suppose first that $I\leq T\leq R$.  If $T\equiv I\pmod2$, choose a complete allocation for which every $c_i(x_i)$ is even.  If $T\not\equiv I\pmod2$, then $I<T<R$.  Choose an allocation with total $T-1$ for which every $c_i(x_i)$ is even.  Since $T-1<R$, some coordinate satisfies $x_i<\rho_i'$.  Both $x_i$ and $\rho_i'$ give even values of $c_i$, so in fact $x_i\leq\rho_i'-2$.  Increase this coordinate by $1$.  The resulting vector is a complete allocation, and exactly one value $c_i(x_i)$ is odd.

Suppose next that $T<I$.  If $\theta\geq1$, the preceding formulas give $c_i(\iota_i)=0$ in every row, so $\iota_i'=\iota_i$ for each new color and
\[
 I=\sum_i\iota_i\leq T,
\]
a contradiction.  Thus $\theta\leq0$.  It follows that every $\iota_i=0$ and every $\iota_i'$ is $0$ or $1$.  Thus $I\leq q(q+2m)$, the number of new colors.  Start with the allocation $x_i=\iota_i'$, whose total is $I$, and choose any $I-T$ coordinates for which $x_i=1$.  Change each of these coordinates from $1$ to $0$.  The resulting vector is a complete allocation, and exactly $I-T$ values $c_i(x_i)$ are odd.  Moreover,
\[
 I-T\leq q(q+2m)-qm(m-1)=q(q+3m-m^2)\leq q^2.
\]

Finally suppose that $T>R$.  For each $i$, let $\varepsilon_i=\rho_i-\rho_i'\in\{0,1\}$, and put $h=\sum_i\varepsilon_i$.  Since $\sum_i\rho_i=R+h$ and $T\leq\sum_i\rho_i$, we have $T-R\leq h$.  Start with the allocation $x_i=\rho_i'$, whose total is $R$.  Choose any $T-R$ indices for which $\varepsilon_i=1$, and for each of them change $x_i$ from $\rho_i'$ to $\rho_i=\rho_i'+1$.  The resulting vector is a complete allocation, and exactly $T-R$ values $c_i(x_i)$ are odd.  The estimate above gives
\[
 T-R
 =h-\left(\sum_i\rho_i-T\right)
 \leq h-2mq
 \leq q(q+2m)-2mq
 =q^2.
\]

Thus in every case there is a complete allocation $x=(x_i)$ for which at most $q^2$ of the values $c_i(x_i)$ are odd.  For each new color $i$, put
\[
 y_i=m-(\alpha\beta)_i-2x_i.
\]
As observed above, the $x_i$ and $y_i$ are nonnegative integers with
\[
 \sum_i x_i=qm(m-1),
 \qquad
 \sum_i y_i=mq(q-1).
\]
We may therefore color, for each new color $i$, exactly $x_i$ currently uncolored $\alpha^2\beta$-edges and exactly $y_i$ currently uncolored $\alpha\beta^2$-edges with $i$.  These totals use every edge of the two types exactly once.  After this coloring,
\[
 (\alpha^2\beta)_i=x_i,
 \qquad
 (\alpha\beta^2)_i=y_i,
\]
and each new color has degree $m$ at $\alpha$.

For a new color $i$, let $b_i=c_i(x_i)$.  By the preceding calculation, $r_i=q-\dg_i(\beta)=2a_i+3b_i$.

If $q\geq9$, at most $q^2$ of the new values $b_i$ are odd, and at most $m^2\leq q^2$ old values are odd.  Since $q^2\leq\binom q3$, the required parity bound follows.

Suppose now that $q<9$.  Under feasibility, with $m\geq5$ and $(m,q)\ne(5,7)$, the only possibilities are $(5,6)$, $(6,6)$, and $(6,8)$.  In all three cases $T\geq R$.  Direct summation of the even-endpoint totals above gives $T-R=18,0,64$, respectively, so the construction above has exactly that many odd new values of $b_i$.  The old colors contribute $6,6,8$ odd values, respectively.  Thus the total numbers of odd $b_i$ are $24,6,72$, which are at most $2\binom63=40$, $2\binom63=40$, and $2\binom83=112$, respectively.

It remains to treat $(m,q)=(5,7)$.  Here $\delta=0$.  Color the seven $\beta$-edges with seven old colors; for these colors take $(a_i,b_i)=(0,2)$, and for the remaining eighteen old colors take $(a_i,b_i)=(2,1)$.  For the new colors, start from the $n\equiv0\pmod3$ table and make the $\alpha\beta$-count-preserving switch $10\times3\longrightarrow6\times5+4\times0$ in the $\alpha\beta$-column.  The resulting allocation is
\[
\begin{array}{c|ccccc}
\text{number}&(\alpha\beta)_i&a_i&x_i&y_i&b_i\\ \hline
6&5&1&0&0&0\\
25&3&0&0&2&0\\
36&0&0&1&3&0\\
52&0&0&2&1&1
\end{array}
\]
Each row has $\alpha$-degree $5$, and the remaining degree needed at $\beta$ is $2a_i+3b_i$.  The four rows contain $119=n^2-m^2$ colors and use all $105=3mq$ edges of type $\alpha\beta$.  Moreover, $\sum_i x_i=140=qm(m-1)$ and $\sum_i y_i=210=mq(q-1)$, so all $\alpha^2\beta$- and $\alpha\beta^2$-edges are used.  The new colors contribute $52$ odd values of $b_i$ and the old colors contribute $18$, giving $70=2\binom73$ in total, while $\sum_i a_i=18\cdot2+6=42\leq3\binom72$.  Thus Lemma~\ref{lem:completion} applies here as well, and the proof is complete.

\end{proof}

\begin{proof}[Proof of Theorem~\ref{thm:hypergraph-embedding}]
Necessity is Lemma~\ref{lem:necessity}.  For sufficiency, the case $m=1$ is immediate.

Suppose $m=2$.  Feasibility gives $n\geq6$, so Lemma~\ref{lem:singleton-factor} provides a one-factorization of $\cc H_n$ whose singleton edges form one factor.  On any two vertices $u,v$, the three copies of $\{u,v\}$ lie in three distinct factors; together with the singleton factor they restrict to the unique one-factorization of $\cc H_2$, up to relabeling colors.  Hence the prescribed factorization of $\cc H_2$ embeds.

Finally suppose $m\geq5$.  Lemma~\ref{lem:partial-coloring} colors every edge except those of types $\beta^2$ and $\beta^3$ and produces the residual data required by Lemma~\ref{lem:completion}; that lemma colors the remaining edges.  This proves Theorem~\ref{thm:hypergraph-embedding}, and hence Theorem~\ref{thm:symmetric-cube-embedding}.
\end{proof}

\section{The order-eight $\operatorname{PSL}(2,7)$ factorization}\label{sec:psl27}

At order eight, the profiles from Lemma~\ref{lem:singleton-factor} admit a much more rigid realization under the natural action of $G=\operatorname{PSL}(2,7)$.  The unique $G$-invariant one-factorization of $K_8^2\cup K_8^3$ has factor type $2+3+3$ and extends to a $G$-invariant one-factorization of $\cc H_8$ with orbit sizes $1,7,28,28$.

\medskip
\noindent\textit{Why order eight is special.}

For $n=8$, Lemma~\ref{lem:singleton-factor} gives the three profiles $(8,0,0)$, $(0,1,2)$, and $(0,4,0)$, with multiplicities $1$, $56$, and $7$, respectively.  After detachment these correspond to factor types $1^8$, $2+3+3$, and $2+2+2+2$.

The factor type $2+3+3$ is special to order eight.  At this order,
\[
 2\binom82=\binom83=56,
\]
so each of the $28$ pairs can be accompanied by two triples while using every triple exactly once.

\medskip
\noindent\textit{The invariant $2+3+3$ factorization.}

Let
\[
 X=\mathbb P^1(\mathbb F_7)=\mathbb F_7\cup\{\infty\},
 \qquad
 G=\operatorname{PSL}(2,7)=\operatorname{SL}(2,7)/\{\pm I\}.
\]
Thus $G$ is the projective special linear group of degree two over $\mathbb F_7$.  We use the standard action of $G$ on the projective line by fractional linear transformations
\[
 x\longmapsto \frac{ax+b}{cx+d},
 \qquad
 \begin{pmatrix}a&b\\c&d\end{pmatrix}\in\operatorname{SL}(2,7),
\]
with the usual conventions at $\infty$.  A factorization will be called $G$-invariant if applying any $g\in G$ to all vertices permutes its factors.  Three particularly useful elements are
\[
 \tau_b(x)=x+b\quad(b\in\mathbb F_7),
 \qquad
 \delta(x)=2x\quad(x\in\mathbb F_7),\qquad \delta(\infty)=\infty,
 \qquad
 \rho(x)=-\frac1x.
\]
All three lie in $G$; for $\delta$ one may use the determinant-one representative $\left(\begin{smallmatrix}4&0\\0&2\end{smallmatrix}\right)$.  Here $\tau_b$ fixes $\infty$, while $\rho$ interchanges $0$ and $\infty$.  Thus $G$ is transitive on $X$, and the translations are transitive on $X\setminus\{\infty\}$.  Consequently the action is two-transitive: any ordered pair of distinct points can be sent to any other.

We will also use $|G|=168$.  Indeed, the first column of a matrix in $\operatorname{SL}(2,7)$ can be chosen in $48$ ways, and after that the second column has seven choices giving determinant one.  Hence $|\operatorname{SL}(2,7)|=336$, and quotienting by $\{\pm I\}$ gives $|G|=168$.

Put $\mathcal Q=\{1,2,4\}$ and $\mathcal N=\{3,5,6\}$.  For the base pair $p_0=\{\infty,0\}$, the subgroup fixing both points individually (the pointwise stabilizer) has order $|G|/(8\cdot7)=3$ by two-transitivity.  It contains $\delta$, which has order three, so it is precisely $\langle\delta\rangle$, with the two orbits $\mathcal Q$ and $\mathcal N$ on $X\setminus p_0$.
The element $\rho$ preserves $p_0$ as a set and interchanges $\mathcal Q$ and $\mathcal N$.  Hence the subgroup preserving $p_0$ as a set (the setwise stabilizer) is transitive on the six points outside $p_0$.  By two-transitivity, the same holds for the setwise stabilizer of every pair.  In particular, $G$ is transitive on the $3$-subsets of $X$: first send a chosen pair of one $3$-set to a chosen pair of the other, and then use the stabilizer of that pair to send the remaining point to the remaining point.

For a pair $p\in\binom X2$, let $A_p$ and $B_p$ be the two orbits of the pointwise stabilizer of $p$ on $X\setminus p$, and put
\[
 \Pi_p=\{p,A_p,B_p\}.
\]
The labels $A_p$ and $B_p$ are interchangeable.  For the base pair,
\[
 \Pi_{\{\infty,0\}}
 =\bigl\{\{\infty,0\},\{1,2,4\},\{3,5,6\}\bigr\}.
\]

\begin{theorem}\label{thm:psl233}
There is a unique $G$-invariant one-factorization of $K_8^2\cup K_8^3$ on vertex set $X$.
It consists of the $28$ factors $\Pi_p$, $p\in\binom X2$.  Thus every factor has type $2+3+3$, every pair occurs in exactly one factor, and every $3$-subset occurs in exactly one factor.
\end{theorem}

\begin{proof}
The family $\{\Pi_p:p\in\binom X2\}$ is $G$-invariant: if $g\in G$, then $g$ sends the subgroup fixing $p$ pointwise to the subgroup fixing $gp$ pointwise, and hence sends the two $3$-point orbits for $p$ to those for $gp$.  Every pair occurs once.  Since $G$ is transitive on the $3$-subsets of $X$, every triple occurs the same number of times among the $3$-parts of the $\Pi_p$.  There are $28\cdot2=56=\binom83$ such occurrences, so every triple occurs exactly once.  This proves existence.

For uniqueness, let $\mathcal F$ be any $G$-invariant one-factorization of $K_8^2\cup K_8^3$.  Every vertex of this hypergraph has degree $7+\binom72=28$, so $\mathcal F$ has $28$ factors.  A partition of eight points using only parts of size two and three has type either $2+3+3$ or $2+2+2+2$.  If $x$ factors had the second type, then counting the $\binom82=28$ pair edges gives $4x+(28-x)=28$, and hence $x=0$.  Thus every factor has type $2+3+3$ and contains a unique pair.

Let $F_0$ be the factor containing $p_0=\{\infty,0\}$.  The subgroup $\langle\delta\rangle$ must preserve $F_0$, because it fixes the unique pair in that factor.  It acts on the two $3$-parts of $F_0$; a group of order three cannot interchange two objects, so it fixes both parts setwise.  The two $\langle\delta\rangle$-orbits on $X\setminus p_0$ are $\mathcal Q$ and $\mathcal N$, and therefore $F_0=\Pi_{p_0}$.  Since $G$ is transitive on pairs and $\mathcal F$ is $G$-invariant, the factor containing every pair $p$ is then forced to be $\Pi_p$.
\end{proof}

Thus the $2+3+3$ structure is forced directly by the pair stabilizers; no orientation of the triples is needed to construct it.

\medskip
\noindent\textit{Orienting the triples.}

We now orient the triples in the canonical factorization.  Write $\chi:\mathbb F_7^\times\to\{\pm1\}$ for the quadratic character, so that $\chi(a)=1$ when $a$ is a square in $\mathbb F_7^\times$ and $\chi(a)=-1$ otherwise.  Represent $t\in\mathbb F_7$ and $\infty$ by
\[
 u_t=\binom t1,
 \qquad
 u_\infty=\binom10.
\]
For distinct $x,y,z\in X$, define
\[
 \varepsilon(x,y,z)
 =\chi\!\left(
 \det(u_x,u_y)\det(u_y,u_z)\det(u_z,u_x)
 \right).
\]

\begin{lemma}\label{lem:psl-sign}
The sign $\varepsilon(x,y,z)$ is well defined and $G$-invariant.  It is preserved by cyclic permutations of $x,y,z$ and reversed by every transposition.
\end{lemma}

\begin{proof}
Replacing $u_x,u_y,u_z$ by $au_x,bu_y,cu_z$ multiplies the determinant product by $(ab)(bc)(ca)=(abc)^2$, so its quadratic character is unchanged.  If $A\in\operatorname{SL}(2,7)$, then $\det(Au,Av)=\det(A)\det(u,v)=\det(u,v)$, which proves $G$-invariance.  A cyclic permutation cycles the three determinant factors.  Interchanging two coordinates multiplies their product by $-1$, and $\chi(-1)=-1$ because $-1=6\in\mathcal N$.
\end{proof}

For finite $x,y,z$ and for triples containing $\infty$, respectively,
\[
 \varepsilon(x,y,z)=\chi((x-y)(y-z)(z-x)),
 \qquad
 \varepsilon(\infty,y,z)=\chi(z-y).
\]
Thus, relative to the ordered pair $(\infty,0)$, the positive third points are $\mathcal Q$ and the negative third points are $\mathcal N$.  By two-transitivity and $G$-invariance, for every pair $p=\{x,y\}$ the two sign classes on $X\setminus p$ are exactly the two blocks $A_p,B_p$ of the canonical factorization.  The sign therefore does not create the $2+3+3$ factorization; it makes a consistent choice of one of the two cyclic orientations of every triple.

For a $3$-set $T$, its three cyclic orderings of positive sign constitute its positive cyclic orientation, and the other three its negative cyclic orientation.  Label the two copies of $T$ in $2K_8^3$ by $T^+$ and $T^-$ accordingly.  Label the three copies of a pair $p$ in $3K_8^2$ by $p^{(1)},p^{(2)},p^{(3)}$.  If the two $3$-parts of $\Pi_p$ are $T_p$ and $U_p$, define
\[
 F_p^+=\{p^{(1)},T_p^+,U_p^+\},
 \qquad
 F_p^-=\{p^{(2)},T_p^-,U_p^-\}.
\]
By Theorem~\ref{thm:psl233}, the $28$ factors $F_p^+$ use every first pair copy and every positive triple copy once, while the $28$ factors $F_p^-$ use every second pair copy and every negative triple copy once.  Each family is a single $G$-orbit, indexed by the $28$ pairs of $X$.

\medskip
\noindent\textit{Completing $\cc H_8$.}

For the remaining pair copy we need a $G$-invariant one-factorization of $K_8$.  For $t\in\mathbb F_7$, put
\[
 M_t=\{\{\infty,t\}\}
 \cup
 \{\{t+a,t+3a\}:a\in\mathcal Q\}.
\]
Since $3\mathcal Q=\mathcal N$, each $M_t$ is a perfect matching.  These seven matchings factor $K_8$.  Indeed, for a finite edge $\{x,y\}$, exactly one of $y-x$ and $x-y$ lies in $\mathcal Q$ because $-1=6\in\mathcal N$.  Order its endpoints so that $y-x\in\mathcal Q$.  Then $a=(y-x)/2\in\mathcal Q$ and $t=x-a$ are uniquely determined and give $\{x,y\}=\{t+a,t+3a\}$.

\begin{lemma}\label{lem:psl-matchings}
The seven matchings $M_t$, $t\in\mathbb F_7$, form a single $G$-orbit.  In particular, their one-factorization of $K_8$ is $G$-invariant.
\end{lemma}

\begin{proof}
The translations $\tau_b$ send $M_t$ to $M_{t+b}$, so all seven matchings lie in the orbit of
\[
 M_0=\bigl\{\{\infty,0\},\{1,3\},\{2,6\},\{4,5\}\bigr\}.
\]
Write these four edges as $e_0,e_1,e_2,e_3$ in the displayed order.  Besides $\delta$, the determinant-one matrix $\left(\begin{smallmatrix}2&1\\2&5\end{smallmatrix}\right)$ gives an element $\sigma\in G$ that interchanges $\infty$ with $1$, $0$ with $3$, $2$ with $6$, and $4$ with $5$.  Thus, on the four edges of $M_0$, $\delta$ fixes $e_0$ and cycles $e_1,e_2,e_3$, while $\sigma$ swaps $e_0,e_1$ and fixes $e_2,e_3$.  These two permutations generate all permutations of the four edges.  Hence the stabilizer of $M_0$ induces $S_4$ on those edges and therefore has at least $24$ elements.  Orbit--stabilizer gives $|G\cdot M_0|\leq168/24=7$.  The seven translates $M_t$ are distinct, so they are exactly the orbit of $M_0$.
\end{proof}

For $t\in\mathbb F_7$, put $F_t=\{p^{(3)}:p\in M_t\}$ and $F_*=\{\{x\}:x\in X\}$.

\begin{theorem}\label{thm:psl-cube}
The factors
\[
 \{F_p^+:p\in\tbinom X2\}
 \cup
 \{F_p^-:p\in\tbinom X2\}
 \cup
 \{F_t:t\in\mathbb F_7\}
 \cup\{F_*\}
\]
form a $G$-invariant one-factorization of $\cc H_8=K_8^1\cup3K_8^2\cup2K_8^3$.  Its $64$ factors form four $G$-orbits with sizes and profiles
\[
\begin{array}{c|c|c}
\text{orbit size}&\text{profile}&\text{factor type}\\ \hline
1&(8,0,0)&1^8\\
7&(0,4,0)&2+2+2+2\\
28&(0,1,2)&2+3+3\\
28&(0,1,2)&2+3+3.
\end{array}
\]
\end{theorem}

\begin{proof}
The two $28$-element families exhaust the first two pair copies and the two oriented triple copies by Theorem~\ref{thm:psl233}.  The seven $F_t$ exhaust the third pair copy by Lemma~\ref{lem:psl-matchings}, and $F_*$ exhausts the singleton edges.  Hence the displayed factors form a one-factorization of $\cc H_8$.

The $G$-invariance of the canonical $2+3+3$ factorization and of the sign shows that $G$ permutes the factors $F_p^+$ among themselves and the factors $F_p^-$ among themselves; since $G$ is transitive on pairs, each family is a $28$-element orbit.  Lemma~\ref{lem:psl-matchings} gives the $7$-element orbit, while $F_*$ is fixed.  The profiles are immediate from the edge sizes in each factor.
\end{proof}

Under the cube--factorization correspondence, this one-factorization determines a symmetric layer-rainbow Latin cube of order $8$ admitting the diagonal action $g:(x,y,z)\mapsto(gx,gy,gz)$ on coordinates, together with the induced action on its $64$ symbols.

Thus the profile multiplicities $1,7,56$ supplied abstractly by Lemma~\ref{lem:singleton-factor} refine under $G$ to orbit sizes $1,7,28,28$; in particular, the $56$ factors of profile $(0,1,2)$ split into two orbits of size $28$.

Chen and Lu~\cite{ChenLu2017} constructed a $\operatorname{PSL}(2,7)$-symmetric factorization of $K_8^3$ into seven spanning $3$-regular factors, each isomorphic to $K_4^3\cup K_4^3$.  Their factorization is not a one-factorization.  Here the same group instead preserves a one-factorization of the non-uniform hypergraph $\cc H_8$, with factor-orbit sizes $1,7,28,28$.

One factor can be seen directly in the cube.  For $p=\{\infty,0\}$, the first pair copy corresponds to the two cells $\{(\infty,\infty,0),(0,0,\infty)\}$.  Since $\varepsilon(1,4,2)=\varepsilon(3,6,5)=1$, the corresponding positive factor occupies
\[
 \begin{gathered}
 (\infty,\infty,0),(0,0,\infty),\\
 (1,4,2),(4,2,1),(2,1,4),\\
 (3,6,5),(6,5,3),(5,3,6).
 \end{gathered}
\]
The three underlying vertex sets $\{\infty,0\}$, $\{1,2,4\}$, and $\{3,5,6\}$ partition $X$, so every coordinate value occurs exactly once in each coordinate position.

\medskip
\noindent\textit{Related classical structures.}

Put
\[
 R_\pm=\{(x,y,z):x,y,z\text{ distinct and }\varepsilon(x,y,z)=\pm1\}.
\]

\begin{lemma}\label{lem:psl-triple-orbits}
The sets $R_+$ and $R_-$ are exactly the two $G$-orbits on ordered triples of distinct points.  For every ordered pair $x\ne y$, exactly three choices of $z$ lie in each orbit.
\end{lemma}

\begin{proof}
A projective transformation fixing three distinct points is the identity.  Indeed, if one fixed point is $\infty$, then $c=0$, so the map has the form $x\mapsto ux+v$; two further fixed points force $u=1$ and $v=0$.  If all three fixed points are finite, the equation $\frac{at+b}{ct+d}=t$ gives a quadratic polynomial with three distinct roots and hence the zero polynomial.  Thus the stabilizer of an ordered triple is trivial, so every orbit has size $|G|=168$.  There are $8\cdot7\cdot6=336$ ordered triples of distinct points.  By Lemma~\ref{lem:psl-sign}, the two sign classes are invariant and are interchanged by a transposition, so they are precisely the two orbits.  Two-transitivity then gives $168/(8\cdot7)=3$ third points in each orbit for every ordered pair.
\end{proof}

The oriented triple structure above has several familiar descriptions, all arising from the same two relations $R_+$ and $R_-$.  Under the action $g(x,y,z)=(gx,gy,gz)$, the triples in $X^3$ fall into six orbits: one in which all three coordinates are equal, three according to which two coordinates are equal, and the two relations $R_+$ and $R_-$ on triples of distinct points.  In the terminology of Mesner and Bhattacharya~\cite{MesnerBhattacharya1990}, these six relations form an association scheme on triples (AST).  Lemma~\ref{lem:psl-sign} says that every $3$-cycle preserves $R_+$ and $R_-$ separately, whereas every transposition interchanges them; this is the skew-symmetric two-class case.

Equivalently, choose from each $3$-set its positive cyclic orientation.  A skew two-graph is such a choice satisfying a four-point parity condition.  The present choice is regular in the sense that every ordered pair has exactly three positive third points.  For distinct $x,y,z,w$,
\[
 \varepsilon(x,y,z)\varepsilon(x,w,y)
 \varepsilon(x,z,w)\varepsilon(y,w,z)=1.
\]
Indeed, inside the quadratic character every unordered pair among the four points occurs twice, so the determinant product is, up to $(-1)^6=1$, the square $\prod_{\{a,b\}\subseteq\{x,y,z,w\}}\det(u_a,u_b)^2$.  Thus an even number of the four induced cyclic orientations have positive sign, exactly the skew-two-graph parity rule.

Finally fix $\infty$.  For $y,z\in\mathbb F_7$, define the Paley tournament by
\[
 y\longrightarrow z
 \quad\Longleftrightarrow\quad
 (\infty,y,z)\in R_+
 \quad\Longleftrightarrow\quad
 z-y\in\mathcal Q,
\]
where the second equivalence follows from $\varepsilon(\infty,y,z)=\chi(z-y)$.
For each $y$, the vertices to which $y$ points are exactly the elements of $y+\mathcal Q$.  The six ordered differences between distinct elements of $\mathcal Q$ are the six nonzero elements of $\mathbb F_7$, each once.  Hence every pair of points lies in exactly one translate $y+\mathcal Q$, so the seven triples $\{y+\mathcal Q:y\in\mathbb F_7\}$ are the lines of the Fano plane.  The same difference calculation shows directly that any two vertices of the tournament have exactly one common vertex to which both point.

Thus the AST, skew two-graph, Paley tournament, and Fano plane are not separate constructions being added to the cube.  They are four classical descriptions of the same oriented $\operatorname{PSL}(2,7)$ structure.  The feature specific to the present setting is that this structure sits inside the unique invariant $2+3+3$ factorization and hence inside the $G$-invariant one-factorization of $\cc H_8$ with orbit decomposition $1+7+28+28$.

We now return to arbitrary feasible embeddings and examine the structure forced when either bound in Theorem~\ref{thm:symmetric-cube-embedding} is attained.

\section{Rigidity at the sharp boundaries}

The results in this section have a simple interpretation in the cube.  When the bound in the case $q=n-m\equiv1\pmod3$ is attained, every cell of the new $q\times q\times q$ subcube with at least two equal coordinates contains an old symbol.  At the other sharp case $q=m$, the old symbols fill the entire new $m\times m\times m$ subcube, so this opposite subcube is itself a symmetric layer-rainbow Latin cube of order $m$ on the old symbols.  Moreover, every new symbol occurs equally often in the three mixed regions having two old and one new coordinate, taken together, as in the three mixed regions having one old and two new coordinates, taken together.

Suppose $q=n-m\equiv1\pmod3$, and consider an embedding of an order-$m$ cube into one of order $n=m+q$.  For an old color $i$, let $s_i,p_i,t_i$ be the numbers of singleton, pair, and triple edges added on the new $q$-set.  Thus $s_i+2p_i+3t_i=q$.  Put
\[
 \Delta=q(3q+1)-4m^2.
\]
By the embedding theorem, $\Delta\geq0$.  If $q=1$, feasibility forces $m=1$; this case is trivial.  Hence in the remainder of this discussion we assume $q\geq4$.

\begin{theorem}\label{thm:defect-stability}
Suppose a one-factorization of $\cc H_m$ embeds in one of $\cc H_n$, put $q=n-m\equiv1\pmod3$, and assume $q\geq4$.  Let $\Delta=q(3q+1)-4m^2$.  With at most $\lfloor\Delta/6\rfloor$ exceptions, the $m^2$ old colors have one of the two profiles
\[
 \left(1,0,\frac{q-1}{3}\right),
 \qquad
 \left(0,2,\frac{q-4}{3}\right).
\]
If $\Delta=0$, exactly $q$ old colors have the first profile and exactly $m^2-q$ have the second.  In this case the old colors use every singleton edge and every pair edge whose vertices lie in the new set.
\end{theorem}

\begin{proof}
For an old color $i$, let $e_i=2s_i+p_i-2$.  Since $s_i+2p_i\equiv1\pmod3$, we have $2s_i+p_i\equiv2\pmod3$, and hence $3\mid e_i$.  Also $e_i\geq0$: this is immediate if $s_i\geq1$, while if $s_i=0$, then $p_i\equiv2\pmod3$ and so $p_i\geq2$.  Equality holds precisely when $(s_i,p_i)=(1,0)$ or $(0,2)$.

Let $S=\sum_i s_i$ and $P=\sum_i p_i$, where the sums are over the old colors.  There are only $q$ singleton edges and $3\binom q2$ pair edges on the new set, so
\[
 2S+P\leq2q+3\binom q2=\frac{q(3q+1)}2.
\]
Therefore $\sum_i e_i=2S+P-2m^2\leq\Delta/2$.  Every positive $e_i$ is at least $3$, and hence at most $\lfloor\Delta/6\rfloor$ old colors have $e_i>0$.

If $\Delta=0$, every $e_i$ is zero and equality holds in the preceding inequality.  Thus all $q$ singleton edges and all $3\binom q2$ pair edges on the new set are used by old colors.  A color of the first profile uses one singleton and no pair edge, so exactly $q$ old colors have that profile; all remaining old colors have the second profile.
\end{proof}

The equality case occurs infinitely often.  We use the following standard fact about Pell equations.  Let $D>0$ be a nonsquare integer, and let $(x_1,y_1)$ be the positive integer solution of
\[
 x^2-Dy^2=1
\]
for which $x_1+y_1\sqrt D$ is smallest.  Then the positive integer solutions are precisely the pairs $(x_k,y_k)$ determined by
\[
 x_k+y_k\sqrt D=(x_1+y_1\sqrt D)^k,\qquad k\geq1;
\]
see, for example,~\cite{NivenZuckermanMontgomery}.

\begin{corollary}\label{cor:pell-boundary}
There are infinitely many admissible pairs $(m,n)$ for which $q=n-m\equiv1\pmod3$ and
\[
 q(3q+1)=4m^2.
\]
The first nontrivial pair is $(m,q,n)=(14,16,30)$.
\end{corollary}

\begin{proof}
The equality $q(3q+1)=4m^2$ is equivalent to
\[
 (6q+1)^2-48m^2=1.
\]
Here $D=48$, and the smallest positive solution is $(7,1)$, so
\[
 6q_k+1+m_k\sqrt{48}=(7+\sqrt{48})^k,\qquad k\geq1.
\]
Multiplying by $7+\sqrt{48}$ gives
\[
 q_{k+1}=7q_k+8m_k+1,
 \qquad
 m_{k+1}=6q_k+7m_k+1.
\]
Therefore, modulo $3$,
\[
 (q_{k+1},m_{k+1})
 \equiv(q_k+2m_k+1,m_k+1).
\]
Since $(q_1,m_1)=(1,1)$, the residue pairs cycle as
\[
 (1,1)\longrightarrow(1,2)\longrightarrow(0,0)
 \longrightarrow(1,1).
\]
Hence every $k\equiv2\pmod3$ gives $q_k\equiv1\pmod3$ and $m_k\equiv2\pmod3$.  Then $n_k=m_k+q_k\equiv0\pmod3$, so both $m_k$ and $n_k$ are admissible.  There are infinitely many such $k$.  Finally, $k=1$ gives the trivial pair $(m,q,n)=(1,1,2)$, while $k=2$ gives $(m,q,n)=(14,16,30)$.
\end{proof}

\medskip
There is another equality case worth isolating.  The counting argument used for necessity gives the following identity for every embedding, and when $q=m$ it forces a second strong restriction.

\begin{lemma}\label{lem:conservation}
Suppose a one-factorization of $\cc H_m$ is embedded in one of $\cc H_n$, and put $q=n-m$.  For the old colors, let $S=\sum_i s_i$, $P=\sum_i p_i$, and $T=\sum_i t_i$, where $s_i,p_i,t_i$ are the numbers of singleton, pair, and triple edges added on the new $q$-set.  Put $D_1=q-S$, $D_2=3\binom q2-P$, and $D_3=2\binom q3-T$.  Then $D_1,D_2,D_3\geq0$ and
\[
 D_1+2D_2+3D_3=q(q-m)(q+m).
\]
\end{lemma}

\begin{proof}
The quantities $D_1,D_2,D_3$ are nonnegative because the old colors cannot use more edges of any type than are available.  Summing $s_i+2p_i+3t_i=q$ over the $m^2$ old colors gives $S+2P+3T=m^2q$.  Since $q+6\binom q2+6\binom q3=q^3$, we obtain $D_1+2D_2+3D_3=q^3-m^2q=q(q-m)(q+m)$.
\end{proof}

\begin{corollary}\label{cor:linear-boundary-rigidity}
Under the hypotheses of Lemma~\ref{lem:conservation}, if $q=m$, then the old colors use every singleton, pair, and triple edge whose vertices lie in the new vertex set.  Consequently no new color uses an edge whose vertices all lie in the new set, and every new color $i$ satisfies $(\alpha^2\beta)_i=(\alpha\beta^2)_i$.
\end{corollary}

\begin{proof}
When $q=m$, Lemma~\ref{lem:conservation} gives
\[
 D_1+2D_2+3D_3=0.
\]
Since $D_1,D_2,D_3\geq0$, all three vanish.  Thus every edge whose vertices all lie in the new set is already used by an old color.  For a new color $i$, the required degrees at $\alpha$ and $\beta$ are both $m=q$.  Since $(\beta)_i=(\beta^2)_i=(\beta^3)_i=0$, the two degree equations are $(\alpha\beta)_i+2(\alpha^2\beta)_i+(\alpha\beta^2)_i=m$ and $(\alpha\beta)_i+(\alpha^2\beta)_i+2(\alpha\beta^2)_i=m$.  Subtracting gives $(\alpha^2\beta)_i=(\alpha\beta^2)_i$.
\end{proof}

\section{Smallest hosts and infinite extensions}

For an admissible $m$, let $h(m)$ be the least admissible integer $n>m$ such that every symmetric layer-rainbow Latin cube of order $m$ embeds in one of order $n$.

\begin{corollary}\label{cor:smallest-host}
For every admissible $m$,
\[
 h(m)=
 \begin{cases}
 2,&m=1,\\
 6,&m=2,\\
 2m,&m\equiv0\pmod3,\ m\geq6,\\
 2m+1,&m\equiv2\pmod3,\ m\geq5.
 \end{cases}
\]
\end{corollary}

\begin{proof}
The cases $m=1,2$ follow directly from Theorem~\ref{thm:symmetric-cube-embedding}, including the exceptional pair $(2,5)$.  Suppose $m\geq5$.  As observed above, feasibility always implies $q=n-m\geq m$, so every containing cube has order $n\geq2m$.  Thus $h(m)$ is the first admissible order at least $2m$.

If $m\equiv0\pmod3$, the order $2m$ is admissible and has $q=m$, so it is feasible.  If $m\equiv2\pmod3$, the order $2m$ is not admissible, while $2m+1$ is admissible and has $q=m+1\equiv0\pmod3$, so it is feasible.
\end{proof}

\begin{remark}
Iterating the corollary gives explicit chains of embeddings.  If $m\equiv0\pmod3$ and $m\geq6$, every cube of order $m$ extends successively through $m,2m,4m,\ldots$.  If $m\equiv2\pmod3$ and $m\geq5$, it extends through $m,2m+1,4m+3,8m+7,\ldots$, with $n_k=2^k(m+1)-1$.  The small orders begin $1\to2\to6\to12\to\cdots$ and $2\to6\to12\to\cdots$.  Taking the union along such a nested chain gives an array on a countably infinite coordinate set and a countably infinite symbol set in which every coordinate layer contains every symbol exactly once.  Thus every finite symmetric layer-rainbow Latin cube is contained in such a countably infinite symmetric layer-rainbow Latin cube.
\end{remark}

\bibliographystyle{amsplain}
\bibliography{symbib_embedding}

\end{document}